%% file: paper.tex
\documentclass[11pt, reqno]{amsart}
\usepackage{color}
\usepackage{comment}
\usepackage[]{hyperref}
\usepackage{amsmath}
\usepackage{amsfonts}
\usepackage{amssymb}
\usepackage{amscd}
\usepackage{tikz}
	\usetikzlibrary{arrows,snakes}
	\usetikzlibrary{decorations.pathreplacing} 
	\usetikzlibrary{matrix}
\usepackage{stmaryrd}
\usepackage{multirow}
\usepackage{mathtools}
\usepackage{subfigure}
\usepackage{verbatim}
\usepackage{pinlabel}
\usepackage[all,cmtip,knot]{xy}
\usepackage{enumerate, enumitem, comment}
\usepackage[normalem]{ulem}

\newtheorem{theorem}{Theorem}[section]

\newtheorem{lemma}[theorem]{Lemma}
\newtheorem{proposition}[theorem]{Proposition}

\newtheorem{corollary}[theorem]{Corollary}

\newtheorem*{namedtheorem}{\theoremname}
\newcommand{\theoremname}{testing}

\theoremstyle{definition}

\theoremstyle{remark}
\newtheorem{remark}[theorem]{Remark}

\def\bbA{\mathbb{A}}
\def\bbB{\mathbb{B}}
\def\F{\mathbb{F}}

\def\bbX{\mathbb{X}}

\def\Z{\mathbb{Z}}
\def\Q{\mathbb{Q}}

\def\cT{\mathcal{T}}

\def\bfh{\mathbf{h}}

\def\bfv{\mathbf{v}}

\def\Fb{\mathbf{F}}
\def\Ab{\mathbf{A}}
\def\Bb{\mathbf{B}}
\def\vb{\mathbf{v}}
\def\hb{\mathbf{h}}
\def\CFb{\mathbf{CF}}
\def\HFb{\mathbf{HF}}

\def\bfD{\mathbf{D}}

\def\mfs{\mathfrak{s}}
\def\mft{\mathfrak{t}}

\def\CFKi{CFK^{\infty}}

\def\gr{\textup{gr}}

\def\mix{\textup{mix}}

\def\red{\textup{red}}
\def\ref{\textup{ref}}

\def\spinc{\textup{Spin}^c}

\def\coker{\operatorname{coker}}

\def\rank{\operatorname{rank}}

\def\d{\partial}

\newcommand{\co}{\mskip0.5mu\colon\thinspace}

\author[Jennifer Hom]{Jennifer Hom}
\address {School of Mathematics, Georgia Institute of Technology, Atlanta, GA 30332}
\address{School of Mathematics, Institute for Advanced Study, Princeton, NJ 08540}
\email{hom@math.gatech.edu}

\author[Tye Lidman]{Tye Lidman}
\address {School of Mathematics, Institute for Advanced Study, Princeton, NJ 08540}
\email {tlid@math.utexas.edu}

\numberwithin{equation}{section}

\title{A note on positive definite, symplectic four-manifolds}

\begin{document}
\maketitle

\begin{abstract}
We prove that a positive definite smooth four-manifold with $b_2^+ \geq 2$ and having either no 1-handles or no 3-handles cannot admit a symplectic structure.
\end{abstract}

\input{intro}

\input{background}

\input{non-symplectic}

\bibliographystyle{amsalpha}
\bibliography{references}

\end{document}

%% file: intro.tex
\section{Introduction}\label{sec:intro}

The geography problem for symplectic four-manifolds asks which signatures and Euler characteristics can be realized by closed, symplectic four-manifolds.  (For the current discussion, manifolds will be closed, simply-connected, and smooth.)  The importance of this problem is (more than) two-fold.  One reason is that many important constructions of exotic smooth four-manifolds make use of a symplectic structure on one of the manifolds.  A well-known example is to take the $K3$ surface, blown-up at one point.  This is a simply-connected, symplectic manifold with odd intersection form, and with $b_2^+ = 3$ and $b_2^- = 20$.  Freedman's theorem \cite{Freedman} implies that $K3 \# \overline{\mathbb{C}P}^2$ is homeomorphic to $\#_3 \mathbb{C}P^2 \#_{20} \overline{\mathbb{C}P^2}$.  By Taubes's theorem \cite{Taubes}, a four-manifold with $b_2^+\geq 2$ and a $\mathbb{C}P^2$ summand cannot be symplectic, as it has vanishing Seiberg-Witten invariants.  Therefore, we have produced an exotic $\#_3 \mathbb{C}P^2 \#_{20} \overline{\mathbb{C}P^2}$.  Understanding the possible topological types of symplectic four-manifolds help us to understand the topological types of exotica that we should attempt to construct.     

Another purpose of the geography problem is to understand the difference between symplectic and complex or K\"ahler manifolds.  One key example of this is the Bogomolov-Miyaoka-Yau inequality \cite{Bogomolov, Miyaoka, Yau1, Yau2}: every complex surface $S$ of general type satisfies $3 \sigma(S) \leq \chi(S)$.  It is an open question as to whether symplectic manifolds satisfy this as well.  In fact, other than $\mathbb{C}P^2$, there are no known simply-connected symplectic four-manifolds on the BMY line, although Stipsicz has constructed a family which is asymptotic to this line \cite{Stipsicz}.  %Moreover, there is no weaker version of this inequality which has been established for symplectic manifolds in general.  

We focus on a potential family of counterexamples to the symplectic BMY inequality: the case where $X$ is positive-definite with $b_2^+ \geq 2$.  This implies by Donaldson's diagonalizability theorem \cite{Donaldson} and by \cite{Freedman} that $X$ is homeomorphic to $\#_n \mathbb{C}P^2$.  By Taubes's theorem, $\#_n \mathbb{C}P^2$ is not symplectic for any $n \geq 2$.  Therefore, a simply-connected, positive definite, symplectic four-manifold would be an exotic $\#_n \mathbb{C}P^2$, for which there are currently none known.  
 
The goal of this note is to prove that under mild conditions, simply-connected, positive definite, symplectic manifolds do not exist.  Recall that a manifold is called {\em geometrically simply-connected} if it admits a handlebody decomposition with no 1-handles.  Of course, this is equivalent to having a decomposition with no 3-handles instead and such a manifold is simply-connected.  

\begin{theorem}\label{thm:main}
Let $X$ be a closed, geometrically simply-connected four-manifold with $b_2^+(X) \geq 2$.  If the intersection form of $X$ is positive definite, then $X$ is not symplectic.
\end{theorem}

In other words, for geometrically simply-connected, symplectic manifolds with $b_2^+(X) \geq 2$, there exists an extremely weak version of a symplectic BMY inequality: $\sigma(X) + 4 \leq \chi(X)$.  

It is natural to ask to what extent the condition on the handlebody decomposition of $X$ is necessary in Theorem~\ref{thm:main}.  It is an open problem as to whether or not every closed, simply-connected four-manifold is geometrically simply-connected \cite[Problem 4.18]{Kirby}.  In fact, it is  unknown if every closed, simply-connected four-manifold admits a {\em perfect Morse function}: a Morse function with the number of critical points equal to the sum of the Betti numbers of $X$; a perfect Morse function necessarily gives a handlebody decomposition with no 1- and no 3-handles.  By the above theorem, a simply-connected, positive definite, symplectic four-manifold with $b_2^+ \geq 2$ would provide a counterexample to these questions.  This idea for constructing four-manifolds without perfect Morse functions is essentially due to Rasmussen \cite{Rasmussen}.  

While $\mathbb{C}P^2$ has $b_2^+ = 1$, it turns out we can use Theorem~\ref{thm:main} to give an interesting characterization of this manifold.  The following was pointed out to us by Tian-Jun Li.  
\begin{corollary}
If a simply-connected, closed four-manifold $X$ has positive-definite intersection form, a perfect Morse function, and a symplectic structure, then it is diffeomorphic to $\mathbb{C}P^2$.  
\end{corollary}
\begin{proof}
Theorem~\ref{thm:main} implies that $b_2^+(X) = 1$.  By assumption, $X$ has a handle decomposition with one 0-handle, one $+1$-framed 2-handle attached along a knot $K$, and one 4-handle.  Because there are no 1- or 3-handles, we see that $S^3_1(K) = S^3$.  By the knot complement theorem \cite{GordonLuecke}, $K$ is the unknot.  Therefore, we have given $X$ the standard handlebody structure for $\mathbb{C}P^2$.      
\end{proof}

Theorem~\ref{thm:main} will be proved by showing that if $X$ has no 1-handles, $b_2^+ \geq 2$, and positive-definite intersection form, then the closed four-manifold invariants $\Phi_{X,\mfs}$, defined by Ozsv\'ath-Szab\'o using Heegaard Floer homology, must vanish for all $\mfs \in \spinc(X)$.  The result will then follow from a theorem of Ozsv\'ath-Szab\'o which states that if $X$ is a symplectic manifold with $b_2^+ \geq 2$, then $\Phi_{X,\mfs} \neq 0$ for a particular $\mfs$ \cite{OSsymplectic}.  (This is the analogue of Taubes's theorem on the non-vanishing of the Seiberg-Witten invariants of symplectic manifolds.)  The stated vanishing result will follow from a technical computation for a particular cobordism map in Heegaard Floer homology.  

\begin{remark}
It is interesting to point out that we were unable to prove Theorem~\ref{thm:main} using the Seiberg-Witten or Donaldson invariants.  Our argument makes use of the knot filtration on the Heegaard Floer chain complex and its relationship with Dehn surgery, which currently do not exist in the gauge-theoretic invariants.  Alternatively, there is likely a proof of Theorem~\ref{thm:main} using the surgery exact triangle from \cite[Theorem 3.1]{OSinteger}, but as far as the authors know, no such analogue of this particular exact triangle has been proved in monopole or instanton Floer homology.  
\end{remark}

The technical computation needed above will also allow us to obstruct certain four-manifolds from admitting symplectic structures with convex boundary.  Recall that the {\em trace} of an integral surgery $S^3_n(K)$ refers to the four-manifold obtained by attaching an $n$-framed 2-handle to the four-ball along $K$.  We write $W_n(K)$ to denote the trace.  

\begin{theorem}\label{thm:stein}
Consider a knot $K$ in $S^3$ such that $d(S^3_1(K)) = 0$.  Then for $n>0$ and any contact structure on $S^3_n(K)$, the trace $W_n(K)$ cannot be a symplectic filling of $S^3_n(K)$.
\end{theorem}

In the above theorem, $d(S^3_1(K))$ denotes the {\em correction term} from Heegaard Floer homology \cite{OSgraded}.  We do not need to distinguish between weak and strong fillability in the theorem, since these are equivalent for rational homology spheres by \cite{OhtaOno}.

Examples of knots with $d(S^3_1(K)) = 0$ are quite numerous.  For example, this includes any smoothly slice knot \cite{Peters} or any negative torus knot \cite[Theorem 6.1]{OwensStrle}.

\begin{remark}\label{rmk:slice}
Theorem~\ref{thm:stein} can be seen to hold if $K$ is smoothly slice without the techniques developed in this paper.  For expository reasons, we will show this after the proof of Theorem~\ref{thm:stein} in Section~\ref{sec:proofs}.    
\end{remark}

\begin{remark}
It follows from work of the first author and Wu \cite{HomWu} that if $d(S^3_1(K)) = 0$, then $\tau(K) \leq 0$.  Therefore, by work of Plamenevskaya \cite{Plamenevskaya}, we have that the maximum Thurston-Bennequin number for any Legendrian representative of $K$ is at most $-1$, and consequently this proves that there is not an ``obvious'' Stein structure on the trace of $+n$-surgery on $K$.  Theorem~\ref{thm:stein} thus rules out, for example, the existence of another knot $K'$ with a Legendrian representative having $tb \geq n+1$ and for which the trace of $+n$-surgery is diffeomorphic to the trace of $+n$-surgery on $K$.  
\end{remark}

\begin{remark}
Note that there exist knots $K$, for example the unknot, such that $d(S^3_{1}(K)) = 0$ and $S^3_{n}(K)$ admits a fillable contact structure for $n > 0$; the knot $8_{20}$ also has $d(S^3_{1}(K)) = 0$ and $S^3_{n}(K)$  admits a tight contact structure with non-vanishing Heegaard Floer contact invariant for all $n > 0$ \cite[Example 1.5]{Golla}. We are grateful to Steven Sivek for pointing out the latter example.  
\end{remark}

\begin{remark}
Similar techniques to those used here were used by Mark and Tosun in the hat theory to study the behavior of the Heegaard Floer contact invariant under rational surgery \cite{MarkTosun}.  Analogously, their results provide obstructions to fillability for certain contact structures on manifolds obtained by rational surgery (see \cite[Corollary 1.4]{MarkTosun}). \\  
\end{remark}

\noindent {\bf Outline:} In Section~\ref{sec:background}, we review the relevant material from Heegaard Floer homology, including the definition of the closed four-manifold invariants.  This is where we establish the main technical result, which is the vanishing of a particular cobordism map in Heegaard Floer homology.  In Section~\ref{sec:proofs}, we prove Theorems~\ref{thm:main} and \ref{thm:stein}. \\ 

\noindent {\bf Acknowledgments:} We would like to thank Tian-Jun Li, Ciprian Manolescu, Tim Perutz, Steven Sivek, and Ian Zemke for helpful discussions.  The first author was partially supported by NSF grants DMS-1128155, DMS-1307879, and a Sloan Research Fellowship.  The second author was partially supported by NSF grants DMS-1128155 and DMS-1148490.

%% file: background.tex
\section{Four-manifold invariants and Heegaard Floer homology}\label{sec:background}
\subsection{The mixed invariants}
We begin by reviewing the Ozsv\'ath-Szab\'o four-manifold invariants, defined in \cite{OSsmoothfour}, coming from Heegaard Floer homology. We will work over $\F = \Z/2\Z$, since the proofs of invariance currently only hold for $\Z/2\Z$-coefficients \cite{Zemke} (see also \cite{Juhasz}).  However, modulo this issue, the arguments in this paper can easily be extended to hold for any coefficients.  

We focus on the definition of the mixed invariants of closed four-manifolds from \cite[Section 9]{OSsmoothfour}; see also \cite[Section 11.3]{MOlink}.  We assume the reader is familiar with Heegaard Floer homology for three-manifolds as in \cite{OSinvariance} and the knot Floer complex as in \cite{OSknots}.  Given a Spin$^c$ cobordism $(W,\mft)$ from $(Y_1,\mfs_1)$ to $(Y_2, \mfs_2)$, Ozsv\'ath and Szab\'o define a map 
\[
F^\circ_{W,\mft} \co HF^\circ(Y_1,\mfs_1) \to HF^\circ(Y_2,\mfs_2), 
\] 
where we use $\circ$ to denote either $\hspace{0.1cm} \widehat{\hspace{.07cm}} \hspace{0.1cm}, +, -,$ or $\infty$.  It is important to note that this map is grading homogeneous, but not necessarily grading preserving.  When $\mfs_1, \mfs_2$ are torsion, the change in degree is computed by 
\[
\gr(F^\circ_{W,\mft}) = \frac{ c_1(\mft)^2 - 2\chi(W) - 3\sigma(W) }{4}.  
\]

Let $X$ be a closed, oriented four-manifold with $b_2^+(X) \geq 2$. We delete two four-balls to obtain a cobordism $W$ from $S^3$ to $S^3$. Cut $W$ along a three-manifold $N$ such that $W$ is divided into cobordisms $W_1$, $W_2$ with $b_2^+(W_1), b_2^+(W_2) \geq 1$ and $\delta H^1(N; \Z)=0 \subset H^2(W, \d W; \Z)$, where $\delta \co H^1(N;\Z) \to H^2(W, \d W;\Z)$ is the coboundary map in the relative Mayer-Vietoris sequence. We call such an $N$ an \emph{admissible cut}, and one always exists.  Note that the latter condition guarantees that there is at most one Spin$^c$ structure on $W$ which restricts to given Spin$^c$ structures on $W_1$ and $W_2$ respectively.  

Let $\mft$ be a Spin$^c$ structure on $X$, and by abuse of notation, let $\mft$ also denote the induced Spin$^c$ structure on $W$. The cobordisms $W_1, W_2$ induce the following maps:
\begin{align*}
	F^-_{W_1, \mft |_{W_1}} &\co HF^-(S^3) \rightarrow HF^-(N, \mft |_N) \\
	F^+_{W_2, \mft |_{W_2}} &\co HF^+(N, \mft |_N) \rightarrow HF^+(S^3),
\end{align*}
each of which factors through $HF_\red (N, \mft |_N)$. We compose these two maps to obtain the mixed map
	\[ F^\mix_{W, \mft} \co HF^-(S^3) \rightarrow HF^+(S^3). \]
%which shifts degree by
%	\[ d(\mft) = \frac{c_1(\mft)^2 - 2\chi(X)-3\sigma(X)}{4}. \]
%Let $\Theta_-$ denote the generator of $HF^-(S^3)$ of maximal grading. The invariant $\Phi_{X, \mft}$ is defined to be
%	\[ \Phi_{X, \mft} = U^{d(\mft)/2} \cdot F^\mix_{W, \mft}(\Theta_-) \in HF^+_0(S^3) \cong \F. \]
We define the closed four-manifold invariant $\Phi_{X,\mft} \in \F$ to be non-zero if and only if the mixed map $F^\mix_{W,\mft}$ is non-zero.  

We remark that the mixed invariants were originally defined with extra structure, taking into account $H_1(X;\Z)$.  Since the manifolds we consider all have trivial first homology, we will not discuss this more general structure.  

The key fact about the mixed invariants that we will need is the non-vanishing result of Ozsv\'ath-Szab\'o for symplectic manifolds, which is the analogue of Taubes's non-vanishing result for the Seiberg-Witten invariants.  While their results are stronger, we will only need the most basic part of the result.  

\begin{theorem}[Ozsv\'ath-Szab\'o, \cite{OSsymplectic}]\label{thm:symplectic-nonvanishing}
Let $X$ be a closed, connected, oriented, smooth four-manifold with $b_2^+ \geq 2$.  If $X$ is symplectic, then $\Phi_{X,\mft} \neq 0$ for $\mft$ the canonical Spin$^c$ structure. 
\end{theorem}

\noindent In fact, we will only need that $\Phi_{X,\mft}$ is non-zero for some $\mft$.  

\subsection{Completions}\label{sec:completions}
In order to prove Theorems~\ref{thm:main} and \ref{thm:stein}, we will need to understand, in certain situations, the Heegaard Floer homology cobordism map associated to a 2-handle attachment.  We will do this computation using the mapping cone formula \cite{OSinteger}, a tool for computing the Heegaard Floer homology of surgery on a knot, which we will recall in the next subsection.  This approach describes the cobordism map in terms of the knot Floer complex, $CFK^\infty$.  We will be particularly interested in computing these cobordism maps for $HF^-$ (as opposed to $HF^+$).  In order to apply the mapping cone formula for the minus theory, we must work with completed coefficients, i.e., we work with modules over the ring $\F[[U]]$  as in \cite{MOlink} (see also \cite{KM}), instead of $\F[U]$.  We do not discuss the case of $HF^+$, since completions do not change $HF^+$ as every element is in the kernel of $U^n$ for some $n$.  

For a Spin$^c$ three-manifold $(Y,\mfs)$, we define $\CFb^-(Y,\mfs)$ analogously to $CF^-(Y,\mfs)$, except we use the base ring $\F[[U]]$ instead of $\F[U]$.  In other words $\CFb^-(Y,\mfs) = CF^-(Y,\mfs) \otimes_{\F[U]} \F[[U]]$.  From this, we obtain $\HFb^-(Y,\mfs)$.  From now on, we will only discuss the case that $\mfs$ is torsion.  (See \cite[Section 2]{MOlink} for the case of non-torsion Spin$^c$ structures.)  It is an annoying observation that $\CFb^-(Y,\mfs)$ is not a relatively $\mathbb{Z}$-graded chain complex in the usual sense, since if the degree of $U$ is $-2$, then $\CFb^-(Y,\mfs)$ does not split as a {\em direct sum} of its graded pieces.  We do have that $\CFb^-(Y,\mfs)$ is a relatively $\Z/2\Z$-graded chain complex.  Since $\F[[U]]$ is flat over $\F[U]$, we have that $\HFb^-(Y,\mfs)$ is isomorphic to $HF^-(Y,\mfs) \otimes_{\F[U]} \F[[U]]$; since $HF^-(Y,\mfs)$ is (non-canonically) isomorphic to a direct sum of modules of the form $\F[U]$ or $\F[U]/U^k$, we can recover $HF^-(Y,\mfs)$ from $\HFb^-(Y,\mfs)$, say as a relatively $\Z/2\Z$-graded $\F[U]$-module.  

The cobordism maps $F^-_{W,\mft}\co HF^-(Y_1,\mfs_1) \to HF^-(Y_2, \mfs_2)$ also have analogues in the completed setting as well, denoted $\Fb^-_{W,\mft}\co\HFb^-(Y_1,\mfs_1) \to \HFb^-(Y_2,\mfs_2)$.  More generally, a $\Z$-grading homogeneous chain map $f\co CF^-(Y_1,\mfs_1) \to CF^-(Y_2,\mfs_2)$ respecting the $\F[U]$-module structure induces a $\Z/2\Z$-grading homogenous $\F[U]$-module map $\mathbf{f}\co\CFb^-(Y_1,\mfs_1) \to \CFb^-(Y_2,\mfs_2)$, and similarly on homology.  In this case, $f_*$ is identically zero (respectively surjective) if and only if $\mathbf{f}_*$ is identically zero (respectively surjective).  We point out that this observation is not true if $\mfs_1$ and $\mfs_2$ are non-torsion. For example, for $n \neq 0$, we have that $HF^-(S^2 \times S^1, \mfs_n) \cong \F[U]/(U^n-1)$, where $c_1(\mfs_n)$ is $2n$ times a generator of $H^2(S^2 \times S^1;\Z)$, while $\HFb^-(S^2 \times S^1, \mfs_n)=0$; see \cite[Section 2]{MOlink}.

\subsection{The mapping cone formula and cobordism maps}
We now recall the mapping cone formula of \cite{OSinteger}, which allows us to compute the Heegaard Floer homology of Dehn surgeries, as well as the cobordism maps for 2-handle attachments.  We will also describe the mapping cone formula for $\HFb^-$ given in \cite{MOlink}.  We state these results for the case of $n$-surgery on a knot in $S^3$ for an integer $n$.  For rational surgeries, see \cite{OSrational}.  
%We state these results only for the case of $+1$-surgery.

Let $\CFKi(K)$ denote the knot Floer complex of $K \subset S^3$. The complex $\CFKi(K)$ is freely generated over $\F[U, U^{-1}]$ and has the structure of a $\Z \oplus \Z$-filtered $\F[U]$-module. For $X$ a subset of $\Z \oplus \Z$, let $CX$ denote the subgroup of $\CFKi(K)$ generated by elements with filtration level $(i, j) \in X$. The group $CX$ will inherit the structure of a chain complex if it is a subquotient complex of $\CFKi(K)$, in which case it also inherits the structure of an $\F[U]$-module.  We consider the following subcomplexes of $\CFKi(K)$:

\begin{align*}
	A_s^- &= C \{ i \leq 0 \textup{ and } j \leq s\} \\
	B_s^- &= C \{ i \leq 0\},
\end{align*}
and the quotient complexes
\begin{align*}
	A_s^+ &= \CFKi(K)/A_s^- = C \{ i > 0 \textup{ or } j > s\} \\
	B_s^+ & = \CFKi(K)/B^- = C \{ i > 0\}.
\end{align*}

The complexes $A^\pm_s$ are quasi-isomorphic to $CF^\pm(S^3_N(K), \mfs_{N,s})$ for $N \gg 0$ and some Spin$^c$ structure $\mfs_{N,s}$.  In particular, $H(A^-_s)$ (respectively $H(A^+_s)$) is non-canonically isomorphic to the sum of $\F[U]$ (respectively $\cT^+=\F[U, U^{-1}]/\F[U]$) together with a finite-dimensional, $U$-torsion module.    

\begin{remark}
In \cite{OSinteger}, $A^+_s$ is defined to be $C \{ i \geq 0 \textup{ or } j \geq s\}$ and $B^+_s$ to be $C \{ i \geq 0\}$. Up to an overall grading shift, which will not matter for our purposes, these two definitions agree. In particular, the definitions of $V^+_s$ and $H^+_s$ below agree with their definitions in \cite{NiWu}.
\end{remark}

The complex $B_s^-$ is naturally identified with $CF^-(S^3)$ and $B_s^+$ with $CF^+(S^3)$. Furthermore, there is a chain homotopy equivalence from $C\{ j \leq 0\}$ to $B_s^-$ and from $C\{j >0\}$ to $B_s^+$.
We have natural chain maps
\begin{align*}
	v_s^- \co A_s^- &\rightarrow B_s^- \\
	v_s^+ \co A_s^+ &\rightarrow B_s^+
\end{align*}
consisting of inclusion and projection respectively. We also have maps
\begin{align*}
	h^-_s \co A_s^- &\rightarrow B_{s+n}^- \\
	h^+_s \co A_s^+ &\rightarrow B_{s+n}^+.
\end{align*}
The map $h^-_s$ consists of inclusion into $C\{j \leq s\}$, followed by the identification of $C\{j \leq s\}$ with $C\{ j \leq 0\}$ (via multiplication by $U^s$), followed by the chain homotopy equivalence from $C\{ j \leq 0\}$ to $C\{ i \leq 0\}$. The map $h^+_s$ consists of quotienting by $C\{ i>0 \textup{ and } j\leq s\}$, followed by the identification of $C\{ j>s\}$ with $C\{j>0\}$ (via multiplication by $U^s$), followed by the chain homotopy equivalence from $C\{ j > 0\}$ to $C\{ i > 0\}$. 

We will also work with the completions with respect to $U$ for the minus theory, yielding complexes $\Ab^-_s$ and $\Bb^-_s$, as well as maps $\bfv^-_s$ and $\bfh^-_s$.   

Define
\[ \bfD^-_n \co \prod_{s \in \Z} \Ab^-_s \rightarrow \prod_{s \in \Z} \Bb^-_s \]
where for $a_s \in \Ab^-_s$,
\[ \bfD^-_n(a_s)=\bfv^-_s (a_s) + \bfh^-_s(a_s). \]
Note that the cone of $\bfD^-$ is a $\Z/2\Z$-graded chain complex.

\begin{theorem}[{\cite[Theorem 1.1 and Theorem 11.3]{MOlink}}]\label{thm:mapping-cone}
The homology of the mapping cone of  $\bfD^-_n$ is isomorphic to $\HFb^-(S^3_n(K))$. Moreover, under this identification, the natural map 
\[
H_*(\Bb^-_s) \to H_*(Cone(\bfD^-_n))
\]
is identified with the map
\[ \HFb^-(S^3) \rightarrow \HFb^-(S^3_n(K)) \]
induced by the natural two-handle cobordism $W_n(K)$ endowed with the $s^\textup{th}$ Spin$^c$ structure (for the identification of the Spin$^c$ structures with the integers as in \cite{MOlink}).
\end{theorem}

\begin{remark}
One can similarly define 
\[ D^+_n \co \oplus_{s \in \Z} A^+_s \rightarrow \oplus_{s \in \Z} B^+_s \]
where $D^+_n(a_s)=v^+_s (a_s) + h^+_s(a_s)$ for $a_s \in A^+_s$. The analogous result to Theorem \ref{thm:mapping-cone} holds for $HF^+$, e.g., the homology of the mapping cone of $D^+_n$ is isomorphic to $HF^+(S^3_n(K))$; see \cite[Theorem 1.1]{OSinteger}. However, for our purposes here, we will only need the mapping cone formula for $\HFb^-$.
\end{remark}

\begin{remark}
The reason we need completions and direct products in the minus theory is that for a given cobordism $W$, the element $F^-_{W,\mft}(x)$ can be non-zero for infinitely many $\mft$; for further details, see \cite[Section 8.1]{MOlink}. This situation never occurs in the plus theory by \cite[Theorem 3.3]{OSsmoothfour}.   
\end{remark}

Theorem~\ref{thm:mapping-cone} also provides an effective way to compute the $d$-invariants of surgery on $K$, which we now describe.    

The chain maps $v^\circ_s$ and $h^\circ_s$ induce maps
\begin{align*}
	v^\circ_{s,*} \co H(A_s^\circ) &\rightarrow H(B_s^\circ) \\
	h^\circ_{s,*} \co H(A_s^\circ) &\rightarrow H(B_{s+n}^\circ).
\end{align*}
Similarly, $\bfv_s^-$ and $\bfh_s^-$ induce maps
\begin{align*}
	\bfv^-_{s,*} \co H(\Ab_s^-) &\rightarrow H(\Bb_s^-) \\
	\bfh^-_{s,*} \co H(\Ab_s^-) &\rightarrow H(\Bb_{s+n}^-).
\end{align*}

We define integers $V^+_s$ and $H^+_s$ as follows. Let $\cT^+=\F[U, U^{-1}]/\F[U]$. Recall that $H(B_s^+) \cong \cT^+$. There is a distinguished submodule $\cT^+ \subset H(A^+_s)$ consisting of elements which are in the image of $U^N$ for arbitrarily large $N$. Define
\[ V^+_s = \rank (\ker v^+_{s,*}|_{\cT^+}) \]
\[ H^+_s = \rank (\ker h^+_{s,*}|_{\cT^+}). \]
Then $v^+_{s,*}|_{\cT^+}$ (respectively $h^+_{s,*}|_{\cT^+}$) is multiplication by $U^{V^+_s}$ (respectively $U^{H^+_s}$).

From \cite[Lemma 2.4 and Proposition 1.6]{NiWu} we have
\begin{equation}\label{eqn:dV}
	d(S^3_1(K)) = -2V^+_0
\end{equation}
\begin{equation}\label{eqn:Vmono}
	V^+_{s+1} \leq V^+_s
\end{equation}
\begin{equation}\label{eqn:VHsym}
	H^+_s=V^+_{-s}.
\end{equation}

We can similarly define
\[ V^-_s = \rank (\coker v^-_{s,*}) \]
\[ H^-_s = \rank (\coker h^-_{s,*}). \]
Recall that $H(B_s^-) \cong \F[U]$ and $H(A_s^-)/\ker(U^N) \cong \F[U]$ for $N \gg 0$.  Since $v^-_s$ and $h^-_s$ are zero on the $U$-torsion submodule of $H(A^-_s)$, we see that $v^-_{s,*}$ (respectively $h^-_{s,*}$) induces a self-map of $\F[U]$ which is given by multiplication by $U^{V^-_s}$ (respectively $U^{H^-_s}$).  

\begin{lemma}
For all $s$, we have $V^-_s=V^+_s$.
\end{lemma}

\begin{proof}
The proof is essentially identical to that of \cite[Proposition 2.13]{OSS}. Namely, we consider the following commutative diagram
\[
\begin{CD}
	H(A^-_s) @>>> HF^\infty(S^3) @>>> H(A^+_s)  \\
	@Vv^-_{s,*} VV	@VVV	@Vv^+_{s,*}VV\\
	HF^-(S^3) @>>> HF^\infty(S^3) @>>> HF^+(S^3),
\end{CD}
\]
where the upper left arrow is induced by inclusion into $CFK^\infty(K)$ and the upper right arrow is induced by projection. Recall that $v^+_{s,*}$ restricted to the image of $HF^\infty(S^3)$ in $H(A^+_s)$ is multiplication by $U^{V^+_s}$.  Thus, the middle vertical arrow above is given by multiplication by $U^{V^+_s}$. Similarly, since both $v^-_{s,*}$ and the map from $H(A^-_s)$ to $HF^\infty(S^3)$ are zero on the $U$-torsion submodule of $H(A^-_s)$, the map given by the middle vertical arrow is multiplication by $U^{V^-_s}$. It follows that $V^-_s=V^+_s$.
\end{proof}

In light of the above lemma, we will omit the superscript and simply write $V_s$ rather than $V^\pm_s$.

\subsection{A vanishing theorem for 2-handle maps}
Using Theorem~\ref{thm:mapping-cone}, we are now able to establish the main technical result we will need to prove Theorems~\ref{thm:main} and \ref{thm:stein}.  

\begin{proposition}\label{prop:vanishing}
Suppose that $d(S^3_1(K)) = 0$. Then for $n$ a positive integer, the induced map 
	\[ F^-_{W_n(K), \mft} \co HF^-(S^3) \to HF^-(S^3_n(K)) \]
is identically zero for all Spin$^c$ structures $\mft$ on $W_n(K)$.
\end{proposition}

\begin{remark}
Rasmussen uses similar arguments to study the cobordism map in the case of $0$-framed $2$-handles in \cite{Rasmussen}.
\end{remark}

\begin{proof}
For ease of exposition, we consider the case $n=1$; a similar argument applies for $n>1$.

As discussed in Section~\ref{sec:completions}, it suffices to show that $\Fb^-_{W_1(K),\mft}$ is zero for all $\mft$ (i.e., to establish the result for completed coefficients).  We now fix $\mft$ and let $s$ be such that $\langle c_1(\mft), [\widehat{\Sigma}] \rangle = 2s + 1$, where $\widehat{\Sigma}$ is a capped-off Seifert surface for $K$ in $W_1(K)$.  Throughout this proof, we only work with the minus theory, so to simplify notation, we write $\Ab_s$ rather than $\Ab^-_s$, etc.  

Consider the quotient complex $\bbX$ of $Cone(\bfD_1)$ defined by $\bbX = \bbA \oplus \bbB$, where 
\[ \bbA = \oplus^{g-1}_{i = 1-g} \Ab_i, \qquad \bbB = \oplus^{g-1}_{i = 2-g} \Bb_i.\]
Here, the differential on $\bbX$ is induced by $\bfD_1$ and $g$ denotes the Seifert genus of $K$.  See Figure \ref{fig:truncatedcone}. By \cite[Lemma 8.8]{MOlink} (see also \cite[Section 4.1]{OSinteger}), we see that the projection from $Cone(\bfD_1)$ to $\bbX$ is a quasi-isomorphism.  Therefore, the cobordism map 
	\[ \Fb^-_{W_1(K), \mft} \co \HFb^-(S^3) \to \HFb^-(S^3_1(K)) \]
will be zero if the the inclusion of $\Bb_{s}$ into $\bbX$ is trivial on homology; indeed, the inclusion of $\Bb_{s}$ into $Cone(\bfD_1)$ would induce the zero map on homology, and the result will follow from Theorem~\ref{thm:mapping-cone}.  Consider the short exact sequence
	\[ 0 \rightarrow \bbB \rightarrow \bbX \rightarrow \bbA \rightarrow 0 \]
and the associated long exact sequence
	\[ \cdots \rightarrow H(\bbA)  \overset{\d}\rightarrow H(\bbB) \overset{\iota}\rightarrow H(\bbX) \rightarrow \cdots \]
The boundary operator $\d$ in the long exact sequence is induced by $\bfD_1$.  If we can show that the boundary operator is surjective, this will imply that the map $\iota \co H(\bbB) \rightarrow H(\bbX)$ is zero. Since the map from $H(\Bb_{s})$ to $H(\bbX)$ factors through $H(\bbB)$, it will follow that $\Fb_{W_1(K), \mft}$ is zero. Note that for $s \leq 1 - g$ or $s \geq g$, the triviality of the cobordism map follows immediately from the fact that the map from $\Bb_s$ into $\bbB$ is identically 0 on the chain level.  We now consider the case $2 - g \leq s \leq g -1$.  

\begin{figure}[ht!]
\begin{tikzpicture}[scale=1]
	\draw (-6,0) node [] (A1-g) {$\Ab_{1-g}$};
	\draw (-4,0) node [] (A2-g) {$\Ab_{2-g}$};
	\draw (-2,0) node [] (A-1) {$\Ab_{-1}$};
	\draw (0,0) node [] (A0) {$\Ab_0$};
	\draw (2,0) node [] (A1) {$\Ab_1$};
	\draw (4,0) node [] (Ag-2) {$\Ab_{g-2}$};
	\draw (6,0) node [] (Ag-1) {$\Ab_{g-1}$};
	
	\draw (-4,-2) node [] (B2-g) {$\Bb_{2-g}$};
	\draw (0,-2) node [] (B0) {$\Bb_0$};
	\draw (2,-2) node [] (B1) {$\Bb_1$};
	\draw (6,-2) node [] (Bg-1) {$\Bb_{g-1}$};

	\draw (-2.75,-1) node [] {$\dots$};	
	\draw (3.25,-1) node [] {$\dots$};
	
	\draw [->] (A1-g) -- (B2-g) node [midway, left, xshift=1pt] {\tiny${\hb_{1-g}}$};
	\draw [->] (A2-g) -- (B2-g) node [midway, left, xshift=2pt] {\tiny${\vb_{2-g}}$};
	\draw [->] (A-1) -- (B0) node [midway, left] {\tiny${\hb_{-1}}$};
	\draw [->] (A0) -- (B0) node [midway, left] {\tiny${\vb_0}$};
	\draw [->] (A0) -- (B1) node [midway, left] {\tiny${\hb_0}$};
	\draw [->] (A1) -- (B1) node [midway, left] {\tiny${\vb_1}$};
	\draw [->] (Ag-2) -- (Bg-1) node [midway, left, xshift=1pt] {\tiny${\hb_{g-2}}$};
	\draw [->] (Ag-1) -- (Bg-1) node [midway, left, xshift=2pt] {\tiny${\vb_{g-1}}$};

\end{tikzpicture}
\caption{}
\label{fig:truncatedcone}
\end{figure}

By hypothesis $d(S^3_1(K))=0$, so by \eqref{eqn:dV} and \eqref{eqn:Vmono}, $V_i=0$ for $i \geq 0$. In particular, for $i \geq 0$, the map $v_{i,*}$ is surjective, and thus $\vb_{i,*}$ is surjective as well.  Similarly, by \eqref{eqn:VHsym}, we have that $H_i = 0$ for all $i \leq 0$, and thus $\hb_{i,*}$ is surjective for all such $i$.  
%For each $H(\Ab_s)$, since this module is isomorphic to $\HFb^-(S^3_n(K), \mfs_n)$ for some $n \gg 0$ and Spin$^c$ structure $\mfs$, we may choose a splitting into $\F [[ U ]] \oplus A^\red_s$ where $A^\red_s$ is a torsion $\F[U]$-module.  Since $H(\Bb_s) \cong \F [[ U ]]$ and $\vb_{s}$ is $U$-equivariant, it follows that $\vb_{s}$ is zero on $A^\red_s$ for all $s$. By \eqref{eqn:VHsym}, we have that $H_s = 0$ for $s \leq 0$, so $\hb_s$ is surjective for all $s \leq 0$.  
Consider the induced map on homology $\bfD_{1,*} \co H(\bbA) \rightarrow H(\bbB)$ given by $a_i \mapsto \vb_{i,*}(a_i)+\hb_{i,*}(a_i)$ for $a_i\in H(\Ab_i)$ where $\vb_{i,*} \co H(\Ab_i) \rightarrow H(\Bb_i)$ and $\hb_{i,*} \co H(\Ab_i) \rightarrow H(\Bb_{i+1})$.  We claim that $H(\Bb_s)$ is in the image of $\bfD_{1,*}$ for $2 - g \leq s \leq g - 1$.  We begin with the case that $0 \leq s \leq g - 1$. Given $b_s \in H(\Bb_s)$, we define an element of $H(\bbA)$ as follows.  Since $\vb_{s,*}$ is surjective, there exists $a_s \in H(\Ab_s)$ such that $\vb_{s,*}(a_s) = b_s$.  Next, for $i \geq s$, define inductively $a_{i+1} \in H(\Ab_{i+1})$ by $\vb_{i+1,*}(a_{i+1}) = \hb_{i,*}(a_{i})$.  Such an $a_{i+1}$ exists by the surjectivity of $\vb_{i+1,*}$.  It is now straightforward to verify that
\begin{align*}
	\bfD_{1,*} \Big(\sum_{i=s}^{g-1} a_i \Big) &= \vb_{g-1,*}(a_i) + \sum_{i=s}^{g-2} \left(\vb_{i,*}(a_i)+\hb_{i,*}(a_i)\right) \\
						&= \vb_{s,*}(a_s) + \hb_{s,*}(a_s) + \vb_{g-1,*}(a_i) + \sum_{i=s+1}^{g-2} \left( \vb_{i,*}(a_{i})+ \hb_{i,*}(a_{i}) \right)\\
						&= \vb_{s,*}(a_s) + \vb_{s+1,*}(a_{s+1}) + \vb_{g-1,*}(a_i) + \sum_{i=s+1}^{g-2} \left(\vb_{i,*}(a_{i}) + \vb_{i+1,*}(a_{i+1})  \right)\\
						&= \vb_{s,*}(a_s) \\
						&= b_s,
\end{align*}
since $\hb_{g-1,*} = \vb_{1 - g,*} = 0$ in the quotient complex $\bbX$, as the targets of these maps are not in $\bbX$. This proves the claim for $0 \leq s \leq g - 1$. A similar argument, using that $\hb_{s,*}$ is surjective for $s \leq 0$, shows that $H(\Bb_s)$ is also in the image of $\bfD_{1,*}$ for $2-g \leq s <0$. This implies that the map $\bfD_{1,*} \co H(\bbA) \rightarrow H(\bbB)$ is surjective, which is what we needed to show.   

The argument is nearly identical for $n>1$. Namely, given $b_s \in H(\Bb_s)$, we use the fact that $v_{s,*}$ (respectively $h_{s,*}$) is surjective for $s \geq 0$ (respectively $s \leq 0$) to find elements $a_i$, $i \equiv s \pmod n$, such that $\bfD_{n,*} (\sum a_i)=b_s$.
\end{proof}

%% file: non-symplectic.tex
\section{Proofs of the main theorems}\label{sec:proofs}

Having established our main technical result in Proposition~\ref{prop:vanishing}, Theorems~\ref{thm:main} and ~\ref{thm:stein} will follow fairly quickly.  Throughout this section, all handlebody decompositions are assumed to have a single 0-handle and if the manifold is closed, a single 4-handle.  For the proof of Theorem~\ref{thm:main}, we will first need a short lemma about handlebody decompositions of definite four-manifolds.  The following is standard, but we include it for completeness.

\begin{lemma}\label{lem:handle-splitting}
Consider a handlebody decomposition of a closed, positive definite four-manifold $X$ with no 1-handles, $n$ 2-handles, and $k$ 3-handles.  Then, there exists a handlebody decomposition of $X$ with the same numbers of 1-, 2-, and 3-handles, but such that the linking matrix for the 2-handles is of the form $\begin{pmatrix} I_{n-k} & 0 \\ 0 & 0_k \end{pmatrix}$, where $I_{n-k}$ denotes the $(n-k) \times (n-k)$ identity matrix and $0_k$ denotes the $k \times k$ zero matrix.
\end{lemma}

\begin{proof}
Consider a handlebody decomposition as in the hypothesis of the lemma.  By construction, the framed link $L$ presenting the 2-handle attachments provides a surgery description for $\#_k S^2 \times S^1$.   Thus, the linking matrix, $Q_L$, for $L$ has nullity $k$.  It is a standard fact in lattice theory that this implies that $Q_L$ is equivalent to a lattice of the form 
\begin{equation}\label{eq:linking-splitting}
Q' = \begin{pmatrix}
A & 0 \\
0 & 0_k
\end{pmatrix},
\end{equation}
where $A$ is a symmetric $(n-k) \times (n-k)$ integral matrix which is invertible over $\Q$.  Since $Q'$ provides a presentation matrix for $H_1(\#_k S^2 \times S^1;\Z) = \mathbb{Z}^k$, we have that $A$ must be unimodular.  Thus, we may handleslide $L$ to obtain a new framed link $L'$ whose linking matrix $Q_{L'}$ is given by $Q'$.  

Let $M \subset L'$ denote the $(n-k)$-component framed sublink whose linking matrix is $A$ (i.e., the first $n-k$ components of $L'$).  Let $Y$ denote the integer homology sphere obtained by framed surgery on $M$.  Consider the decomposition $X = W \cup_Y Z$, where $W$ consists of the 0-handle and the 2-handles attached along the components of $M$ and $Z$ consists of the 2-handles in $L' - M$ (i.e., the ones which have framing 0 and linking number 0 with all other components), together with the 3- and 4-handles.  Note that $b_2(W) = n - k = b_2(X)$.  Since $Y$ is an integer homology sphere, the intersection form of $X$ splits as a sum of the intersection forms for $W$ and $Z$.  It follows that the intersection form of $Z$ is trivial and thus the intersection form of $W$ agrees with that of $X$.  

Since $X$ is positive-definite, Donaldson's theorem implies that the intersection form of $W$ is diagonalizable.  Therefore, there exists a sequence of handleslides of $M$ which results in a framed link whose linking matrix is given by $I_{n-k}$.  Note that since the linking number of any component of $L' - M$ with each component of $M$ is zero, we can perform these handleslides in the entire framed link $L'$ without affecting any framings of or linking numbers involving the components of $L' - M$.  In other words, there exists a sequence of handleslides of $L'$ such that the resulting linking matrix is of the form $\begin{pmatrix} I_{n-k} & 0 \\ 0 & 0_k \end{pmatrix}$.
\end{proof}

\begin{proof}[Proof of Theorem~\ref{thm:main}]
Our goal is to identify a copy of $W_{1}(K)$ in $X$ for some $K$ with $d(S^3_{1}(K)) = 0$ so we can apply Proposition~\ref{prop:vanishing}.  After possibly turning $X$ upside-down, we can find a handlebody decomposition of $X$ with no 1-handles.  Now, we may apply Lemma~\ref{lem:handle-splitting} to find a submanifold of the form $W_{1}(K)$ smoothly embedded in $X$ for some $K$ (i.e., using the handlebody decomposition coming from the lemma, we can attach to the 0-handle one of the $+1$-framed 2-handles).  

Then, we may decompose $X = W_{1}(K) \cup_{S^3_{1}(K)} P$.  Note that $S^3_{1}(K)$ is an integer homology sphere and $b_2^+(W_{1}(K)) = 1$, which implies that $b_2^+(P) = b_2^+(X) - 1 \geq 1$.  In particular, $S^3_{1}(K)$ gives an admissible cut.  Since $S^3_{1}(K)$ is the boundary of $W_{1}(K)$, which is positive definite, we see that $d(S^3_{1}(K)) \leq 0$ \cite[Corollary 9.8]{OSgraded}.  Since $-S^3_{1}(K)$ is the boundary of $P$, which is also positive definite, we have that $d(-S^3_{1}(K)) \leq 0$.  This implies that $d(S^3_{1}(K)) = 0$, since $d$-invariants reverse sign under orientation reversal \cite{OSgraded}.  

Let $\mft$ be a Spin$^c$ structure on $X$, and let $\mft_1$ and $\mft_2$ denote the  restrictions to $W_{1}(K) - B^4$ and $P - B^4$ respectively. Proposition~\ref{prop:vanishing} implies that $F^-_{W_{1}(K) - B^4, \mft_1}=0$. Since the mixed map is equal to $F^+_{P - B^4, \mft_2} \circ F^-_{W_{1}(K) - B^4, \mft_1}$, we have that $\Phi_{X,\mft} = 0$.  The result now follows from Theorem~\ref{thm:symplectic-nonvanishing}.  
\end{proof}

\begin{proof}[Proof of Theorem~\ref{thm:stein}]
Assume that the trace $W_{n}(K)$ symplectically fills $(S^3_n(K),\xi)$, for some contact structure $\xi$ on $S^3_n(K)$. We now apply a standard argument used in \cite{KMpropertyP}, \cite{KMOS}, \cite{OSgenus}.  By \cite{EtnyreHonda} (see also \cite{Eliashberg}, \cite{Etnyre}), we can embed $W_{n}(K)$ in a symplectic manifold $X$ with $b_2^+(X) \geq 2$.  By Theorem~\ref{thm:symplectic-nonvanishing}, we have that $\Phi_{X,\mft} \neq 0$ for some $\mft$.  By construction, we have an admissible cut of $X$ along $S^3_{n}(K)$ with one piece being $W_{n}(K)$.  By assumption $d(S^3_{1}(K)) = 0$, so Proposition~\ref{prop:vanishing} implies that $\Phi_{X,\mft} = 0$, which is a contradiction.  
\end{proof}

We conclude the paper by providing a more elementary proof of Theorem~\ref{thm:stein} in the case that $K$ is smoothly slice.  If $K$ is slice, then by capping off the slice disk in the four-ball with the core of the $n$-framed 2-handle, we obtain a smoothly embedded sphere in $W_n(K)$ with self-intersection $n > 0$.  If the trace provided a symplectic filling of the boundary, we could embed it in a closed symplectic manifold with $b_2^+ \geq 2$ as in the proof of Theorem~\ref{thm:stein}.  However, a closed symplectic four-manifold with $b_2^+ \geq 2$ cannot contain a sphere of positive self-intersection by \cite[Proposition 10.1]{MST}, which is a contradiction.  For a self-contained way to see this last claim, take a regular neighborhood of the sphere, which is a disk bundle over $S^2$ with Euler number $n$, and thus has boundary $L(n,1)$.  Since $n > 0$, this disk bundle has $b_2^+ = 1$, so the lens space gives an admissible cut.  However, $HF_\red(L(n,1)) = 0$, and we see that the Ozsv\'ath-Szab\'o invariants for the four-manifold vanish, contradicting Theorem~\ref{thm:symplectic-nonvanishing}.